\documentclass[11pt,twoside]{amsart}
\usepackage{mathtools,amsmath,amssymb}
\usepackage{mathrsfs}

\usepackage{libertine}       
\usepackage[a4paper,top=1in, bottom=1in, left=1in, right=1in]{geometry} 
\usepackage{graphicx} 
\usepackage[colorlinks=true,linkcolor=blue,citecolor=black]{hyperref} 
\usepackage{cleveref}
\usepackage[backend=biber,backref=true,doi=false,url=false,maxnames=9]{biblatex} 
\renewbibmacro*{pageref}{%
  \iflistundef{pageref}
    {}
    {\addspace
     \mkbibbrackets{\printlist[pageref][-\value{listtotal}]{pageref}}}} 
\usepackage{lipsum}
\usepackage{adjustbox} 
\usepackage{tikz-cd}
\usepackage{mathptmx}
\usepackage{quiver}
\usepackage{float}
\usepackage{amssymb}
\usepackage[english]{babel}
\usepackage{amsthm}
\usepackage{amsmath}
\usepackage{mathtools}
\usepackage{soul}
\usepackage{nicematrix}
\usepackage{array}   
\usepackage{booktabs}
\usepackage{float}
\usepackage{orcidlink}
\usepackage{nomencl}

\makenomenclature
\usepackage[toc,page]{appendix}
\usepackage{caption}
\usepackage{mathrsfs}     
\newcommand{\GL}{\mathrm{GL}}
\newcommand{\M}{\mathrm{M}}

\newcommand{\Z}{\mathbb{Z}}

\newcommand{\R}{\mathscr{O}}

\newcommand{\Zen}{\mathscr{Z}}
\renewcommand{\O}{\mathscr{O}}
\newcommand{\is}{\mathrm{iso}}
\newcommand{\Min}{\mathrm{Min}}
\usepackage{thmtools, thm-restate}
\declaretheorem{theorem}
\newtheorem{lemma}{Lemma}[section]
\newtheorem{proposition}[lemma]{Proposition}
\newtheorem{corollary}[lemma]{Corollary}
\theoremstyle{definition} 

\newtheorem{example}[lemma]{Example}
\newtheorem{remark}[lemma]{Remark}
\newtheorem{question}{Question}
\title[Lifting solutions of polynomial equations on matrices]{Lifting solutions of polynomial equations on matrices over field to complete local principal ideal rings}
\date{\today}
\author[Panja]{Saikat Panja\orcidlink{0000-0002-9639-3122}}
\email[(Panja)]{panjasaikat300@gmail.com}
\address{Indian Statistical Institute, Bengaluru Centre, 8th Mile, Mysore Rd, RVCE Post, Gnana Bharathi, Bengaluru, Karnataka 560059, India}
\author[Roy]{Ayon Roy}
\email[(Roy)]{ayonroy1999@gmail.com}
\address{Indian Institute of Science Education and Research Pune, Dr. Homi Bhabha Road, Pashan, Pune 411 008, India}
\author[Singh]{Anupam Singh}
\email[(Singh)]{anupamk18@gmail.com}
\address{Indian Institute of Science Education and Research Pune, Dr. Homi Bhabha Road, Pashan, Pune 411 008, India}
\thanks{Panja is supported by an NBHM postdoctoral fellowship, file number ending at R\&D-II/6746. Roy is supported by an IISER Pune PhD Fellowship. Singh is supported by an ANRF-MATRICS grant.}
\date{\today}
\subjclass[2020]{15B33,16S50,16S85}
\keywords{polynomial maps, matrix algebras, local rings, lifting problem}
\begin{document}
\maketitle
\begin{abstract}
Let $\widehat{\mathscr O}$ be a complete local principal ideal ring with residue field $k$ of characteristic not $2$ and $f\in \widehat{\mathscr O}[x_1,x_2,\dots,x_m]$. Take $A\in \mathrm M_n(\widehat{\mathscr O})$ with its reduction $\overline{A}\in \mathrm M_n(k)$.
In this article, we study the following lifting problem.
Suppose there exists a tuple
$(\widetilde{B}_1, \widetilde{B}_2, \dots,\widetilde{B}_m)\in \mathrm M_n(k)^m$ of pairwise commuting matrices such that
$f(\widetilde{B}_1, \widetilde{B}_2, \dots,\widetilde{B}_m) = \overline{A}$; under what conditions can this solution be lifted to a tuple
$(B_1,B_2,\dots,B_m)\in \mathrm M_n(\widehat{\mathscr O})^m$ of pairwise commuting matrices satisfying
$f(B_1,B_2,\dots,B_m)=A$ ?
For $\overline{A}$ regular semisimple, we show that, under suitable hypotheses analogous to those appearing in Hensel’s lemma, such a lifting is always possible.
\end{abstract}
\section{Introduction}\label{sec:intro}

Lifting phenomena occur across many areas of mathematics, such as algebraic geometry, number theory, and representation theory.
For example, the Oort conjecture concerns the lifting of a smooth curve in characteristic $p$, together with its group of automorphisms, to a curve, along with the automorphisms, to characteristic zero; see, for example, \cite{ObusWewers2014,Pop2014}.
Another instance arises in the lifting of packets of representations of $G(k)$ to packets for $\widetilde{G}(k)$, where $\widetilde{G}$ is a connected reductive group over a finite field $k$, and $G$ denotes the connected component of the group of $\varepsilon$-fixed points of $\widetilde{G}$. Here $\varepsilon$ is a semisimple $k$-automorphism of $\widetilde{G}$ of finite order; see, for example,
\cite{AdlerLansky2014,AdlerCasselLansky2016,AdlerLansky2023}.
The lifting of a two-dimensional continuous representation
$\overline{\rho}\colon G_{\mathbb Q}\rightarrow \mathrm{GL}_2(k)$
of the absolute Galois group of $\mathbb Q$ to a representation
$\rho\colon G_{\mathbb Q}\rightarrow \mathrm{GL}_2\bigl(W(k)\bigr)$
that is unramified outside a finite set of primes has been studied in \cite{Ravi1999}, under the assumption that $\overline{\rho}$ is odd and absolutely irreducible. Here $k$ is a finite field of characteristic $p$, and $W(k)$ denotes its ring of Witt vectors.

The classical Hensel's Lemma concerns lifting a factorization modulo a prime number $p$ of a polynomial $f$ over the $\Z$ to a factorization modulo $p^m$ for any integer $m\geq 1$, and to a factorization over the $p$-adic integers. This can be generalized easily to the case where the integers are replaced by any commutative ring, the prime number is replaced by a maximal ideal, and the $p$-adic integers are replaced by the completion with respect to the maximal ideal. However, lifting a solution for equations over a non-commutative ring remains relatively unexplored, mainly because it may not always have a positive solution. Take, for example, the ring $\R=\Z/9\Z$ and 
$A=\begin{pmatrix}5&0\\0&2\end{pmatrix}\in\M_2(\R),\,L=2.$
Letting $\Phi_2:\M_2(\R_2)\longrightarrow\M_2(\R_2)$, the squaring map, we have 
\begin{align*}
\overline{\Phi}_2\!\left(\begin{pmatrix}0&1\\2&0\end{pmatrix}\right)=\overline{A}\in\M_2(k),
\end{align*}
and
\begin{align*}
\overline{\Phi}_2^{-1}(\overline{A})
=\left\{
\begin{pmatrix}0&2\\1&0\end{pmatrix},
\begin{pmatrix}0&1\\2&0\end{pmatrix},
\begin{pmatrix}1&1\\1&2\end{pmatrix},
\begin{pmatrix}1&2\\2&2\end{pmatrix}
\right\}.
\end{align*}
However, there is no $B\in\M_2(\R)$ such that $B^2=A$; see \cite[Example 6.1]{PRS2025}.
Our motivation for investigating the lifting problem for solutions of polynomial equations on matrix algebra over a field to matrix algebra over a complete local principal ideal ring comes from the growing interest in understanding polynomial maps on algebras {(for example \cite{KanelMalevRowen2012,Bresar2020,BresarVolci2025})} and various conjectures in the subject. Let us briefly recall the concept of polynomial maps on algebras here; let $\mathscr R$ be a ring with unity, then given an $\mathscr R$-algebra $\mathscr A$, and an element $f\in \mathscr R[x_1,x_2,\ldots,x_n]$ (the free $\mathscr R$-algebra on $n$ generators), one gets a map
\begin{align*}
    \widetilde{f}\colon \mathscr A^n\rightarrow\mathscr A,\,(a_1,a_2,\ldots,a_n)\mapsto f(a_1,a_2,\ldots,a_n),
\end{align*}
known as \emph{polynomial map}. 
In recent years, many interesting results have been proven in this area. As the literature is rapidly expanding, we mention only a few representative instances, while acknowledging that several fundamental contributions remain unmentioned.

A positive solution to the L'vov–Kaplansky conjecture for $\M_2(k)$, where $k$ is a quadratically closed field, was obtained in \cite{KanelMalevRowen2012}.
The images of general polynomials on upper triangular matrices were determined in \cite{GargateMello2022,PanjaPrasad2023}.
The Waring problem for matrices has been studied in
\cite{BresarSemrl2023,BresarSemrl2023b,KishoreSingh2025,PanjaSainiSingh2025Surj}.
A surprising connection between commutators in matrix rings and images of noncommutative polynomials on matrix rings was established in \cite{Bresar2020}.
More general investigations involving rational functions were introduced in \cite{BresarVolci2025}.
Finally, polynomial maps with constants taken from the algebra itself were studied in
\cite{PanjaSainiSingh2025constant,PanjaSainiSingh2025splitOctonion}.

We consider a natural question as follows: let $\R$ be a local ring and consider the algebras $\R[x_1,x_2\cdots,x_n]$ of $n$-variable polynomials and $\M_r(\R)$ of $r\times r$ matrices with entries from $\R$. Let $\mathfrak m$ be the maximal ideal of $\R$. Note that for a polynomial map $f:\M_r(\R)^n\longrightarrow\M_r(\R)$, we get a natural map $\overline{f}:\M_r(k)^n\longrightarrow\M_r(k)$, which further produces a commutative diagram taking into account the canonical map $\theta:\R\longrightarrow k$
\[\begin{tikzcd}
	{\M_r(\R)^n} & {\M_r(\R)} \\
	{\M_r(k)^n} & {\M_r(k)}
	\arrow["f", from=1-1, to=1-2]
	\arrow["\theta"', from=1-1, to=2-1]
	\arrow["\theta", from=1-2, to=2-2]
	\arrow["{\overline{f}}"', from=2-1, to=2-2].
\end{tikzcd}\]
Then it is natural to ask the following question:
\begin{question}
    Let $A\in\M_r(\R)$ and $f\in\R[x_1,x_2,\cdots,x_n]$. Suppose there exists a solution $(C_1,C_2,\ldots,C_n)\in\M_r(k)^n$ such that $\overline{f}(C_1,C_2,\ldots,C_n)=\overline{A}$. Does there exist $(B_1,B_2,\ldots,B_n)\in\M_r(\R)^n$ such that $\theta({B_i})=C_i$ for all $1\leq i\leq n$ and $f(B_1,B_2,\cdots,B_n)=A$?
\end{question}
Of course, such a lifting is not always possible, as described in the discussed example. We give sufficient conditions, both on the matrix $A$ and the tuple $(B_1, B_2,\ldots,B_n)$, so that the solution can be lifted; this is discussed in \Cref{lift-length two} and \Cref{thm:n-var} in the next section.
\subsection{Notations}
Let us fix some notation.
A commutative ring $\mathscr{R}$ with unity is said to be \emph{complete with respect to an ideal $I$} if the canonical map
$\mathscr{R}\longrightarrow \varprojlim_{j\geq 1}\mathscr{R}/I^j\mathscr{R}$
is an isomorphism.
We denote by $\widehat{\O}$ a local principal ideal ring that is complete with respect to its unique maximal ideal $\mathfrak{m}=(\pi)$.
Its residue field $k$ has characteristic not equal to $2$.
For $\ell\ge 1$, the notation $\O_\ell$ is reserved for the ring
$\O_\ell=\widehat{\O}/\pi^{\ell}\widehat{\O}$.
This is a local principal ideal ring of length $\ell$ with unique maximal ideal $\mathfrak{m}_\ell=(\pi_\ell)$, where $\pi_\ell=\pi+(\pi^\ell)$.
The canonical map $\O_{\ell+1}\longrightarrow \O_\ell$ is denoted by $\theta_\ell$.
The natural surjection $\widehat{\O}\rightarrow \O_\ell$ is denoted by $\widehat{\theta}_\ell$.
If $\widehat{\theta}_{\ell+1}(a)=a_{\ell+1}$ for some $a\in\widehat{\O}$, then we write
$a_\ell:=\theta_\ell(a_{\ell+1})\in \O_\ell$.
Thus,
$
\theta_\ell\circ \widehat{\theta}_{\ell+1}=\widehat{\theta}_\ell.
$
From now on, all reductions from $\widehat{\O}$ to the level $\O_\ell$—for elements, matrices, polynomials, and so on—will be denoted uniformly by $\widehat{\theta}_\ell$.
For instance, if $H(t)\in \widehat{\O}[t]$, then
$
\widehat{\theta}_\ell(H(t))=H_\ell(t).
$
In particular, $H_1(t)$ is denoted by $\overline{H}(t)$ to emphasize that it is a polynomial with coefficients in the residue field $k$.
We adopt this notation throughout.
Also by abuse of notations, $\theta$ denotes the map to the field $k$ from the local ring $\R_\ell$ for any $\ell$.
An analogous convention applies to matrices; for example, $\overline{X}=X_1$.
For a commutative ring $\mathscr{R}$ with unity and a matrix $A\in\M_n(\mathscr{R})$, we denote by $\mathscr{Z}_{\M_n(\mathscr{R})}(A)$ the centralizer algebra of $A$.
For a matrix $A\in \M_n(\mathscr R)$, we define $N_A^{\mathscr R}=\{F(t)\in\mathscr R[t]:F(A)=0\}$, the \emph{null ideal of $A$}; see \cite[Section 3.1]{PRS2025} for further details.
A matrix $A\in\M_n(\mathscr R)$ is called \emph{cyclic} if $\overline{A}\in\M_n(k)$ is cyclic. 
Recall that a matrix $X\in\M_n(k)$ is cyclic if the $k[t]$ module $k^n$ via the action of $X$ is a cyclic module. One also has the following: if $A\in\M_n(\R)$ be such that $\overline{A}\in\M_n(k)$ is cyclic, the $\R[t]$ module $\R^n$ via the action of $A$ is a cyclic module as well; see \cite[Lemma 5.2]{PRS2025}.

\color{black}

\section{Lifting solutions of polynomial maps}\label{sec:finite-lifting}

In the article \cite{PRS2025}, the authors produced a strategy to lift a solution of $X^L=\Bar{A}$ over residue field level to the solution of the given equation $X^L=A$ over local ring level for $\gcd(L,p)=1$; where $p$ is the characteristic of the residue field which is an odd prime, and $A$ is cyclic elements of $\M_n(\R_2)$. 
They used a single variable monic polynomial evaluated at a matrix variable with coefficients from $\R_2$ and proved the following proposition.

\begin{proposition}\label{prop:hens-root-full}
Let $\R_2$ be a finite local principal ideal ring of length two with maximal ideal $\mathfrak m= \langle \pi \rangle$ and residue field $k$ of characteristic $p>2$. 
Let $A\in \M_n(\R_2)$ be cyclic, and let $F(t)\in \R_2[t]$ be monic of degree $d$. Suppose there exists $\widetilde{B} \in \M_n(k)$ such that $
\overline{F}(\widetilde{B}) = \overline{A}$ and 
$\overline{F}'(\widetilde{B})\in \GL_n(k)$ where $F'(t)$ is the formal derivative of $F(t)$ in $\R_2[t]$. 
Then there exists $B\in \M_n(\R_2)$ with $\overline{B}=\widetilde{B}$ and $F(B)=A$.
\end{proposition}
This gives a general setup for lifting solutions of such matrix equations from the residue field to the length two level. 
But now we want to see it in some bigger context. 
First, let us investigate for a given matrix equation $F(X)=A$; $A\in M_n(\mathcal{O}_{\ell})$ be cyclic; whether a solution of the reduced equation from length two level (if exists) can be lifted to length $3$ level and so on for a given matrix $A\in M_n(\mathscr{O}_{\ell})$. 
We start with $\mathscr{O}_{\ell}=\mathbb{Z}/p^{\ell}\mathbb{Z}$; where $p$ is an odd prime. 
Note that the ring of $p$-adic integers can be written as $\mathbb{Z}_p=\varprojlim \mathbb{Z}/p^{\ell}\mathbb{Z}$.
In the construction of the lifting strategy, the null ideal played a crucial role. 
Therefore, let us prove the results regarding null ideals and uniqueness of minimal polynomials for a cyclic element $A\in \GL_n(\mathcal{O}_3)$ using the information of $\mathcal{O}_2$ level.

\[\begin{tikzcd}
	{\mathbb{Z}/p^{3}\mathbb{Z}} & {\mathbb{Z}/p^{2}\mathbb{Z}} \\
	& {\mathbb{F}_p}
	\arrow["{\theta_{2}}", from=1-1, to=1-2]
	\arrow["{\theta_1\circ \theta_2}"', from=1-1, to=2-2]
	\arrow["{\theta_1}", from=1-2, to=2-2]
\end{tikzcd}\]

Let $F(t)\in \mathscr{O}_3[t]$ be such that $F(t)\in N_A^{\R_3}=\{G\in \O_3[t]:G(A)=0\}$. 
Then $F(t)=H(t)\chi(A)(t)+R(t)$ where $R(t)=0$ or $\deg(R(t))<\deg (\chi(A)(t))$. If possible, let $R(t)\neq 0$.
If at least one of the coefficients of $R(t)$ is a unit in $\mathscr{O}_3$, it contradicts the minimality of $\chi(\Bar{A})(t)$ in $k[t]$.
Therefore $R(t)\in \mathfrak{m}[t]$. 
Note that $\ker (\theta_{2})=p^2\mathbb{Z}/p^3\mathbb{Z}$. Now we can always write $R(t)=p S(t)$ with a leading coefficient that is not congruent to $0$ $\pmod{p^2}$. 
Therefore $\deg(R(t))=\deg(S(t))$.
As $R(A)=0$ implies $S(A)=p^2Q$ for some $Q\in \M_2(\mathscr{O}_3)$ therefore $\theta_{2}(S(A))=\mathbf{0}$; consequently $\theta_{2}(S)\in N^{\R_2}_{\theta_{2}(A)}$.
Now $N^{\R_2}_{\theta_{2}(A)}=\langle \chi(\theta_{2}(A)) \rangle$ by \cite[Lemma 3.1]{PRS2025}.
Note that $$\deg (\theta_{2}(S))=\deg(S)=\deg(R)<\deg(\chi(A))=\deg(\chi(\theta_{2}(A))),$$
which is a contradiction. Hence, $R(t)$ must be the zero polynomial.
Let us now prove the statement in generality we seek to apply to prove the main theorems.
\begin{proposition}{\label{Null ideal}}
Let $\widehat{\mathscr{O}}$ be a local principal ideal ring, complete with respect to its maximal ideal $\mathfrak{m}=(\pi)$, and its residue field $k$ has characteristic different from $2$.
Let $A\in \GL_n(\widehat{\mathscr{O}})$.
Assume that there exists a monic polynomial $F(t)\in N_A^{\widehat{\R}}$ such that
\begin{align*}
\deg\bigl(\overline{F}(t)\bigr)=\deg\bigl(\Min_{k,\overline{A}}(t)\bigr).
\end{align*}
Then, for every $i\in\mathbb{N}$, the null ideal $N^{\R_{i+1}}_{A_{i+1}}$ is principal and is generated by $F_{i+1}(t)$, where
\begin{align*}
A_i=\widehat{\theta}_i(A), \qquad
F_i(t)=\widehat{\theta}_i\bigl(F(t)\bigr), \qquad
\mathscr{O}_i=\widehat{\mathscr{O}}/\pi^i\widehat{\mathscr{O}}.    
\end{align*}
\end{proposition}
\begin{proof}
   We prove the proposition by using induction on $i$.

   For $i=1$; invoking \cite[Lemma 3.1]{PRS2025}, we obtain $N_{A_2}^{\R_2}=\langle F_2(t)\rangle$. Because then $\mathscr{O}_2$ is a local principal ideal ring of length 2.

   \textbf{Induction Hypothesis:} Let the statement be true for $i=2,3,\cdots,m$.

   \textbf{Inductive step:} Consider $i=m+1$. See the following diagram 
   \[\begin{tikzcd}
	{\mathscr{O}_{m+2}} & {\mathscr{O}_{m+1}} \\
	& {\mathscr{O}_{m}}
	\arrow["{\theta_{m+1}}", from=1-1, to=1-2]
	\arrow["{\theta_{m}\circ\theta_{m+1}}"', from=1-1, to=2-2]
	\arrow["{\theta_{m}}", from=1-2, to=2-2]
\end{tikzcd}\]

Take any polynomial $G(t)\in \mathscr{O}_{m+2}[t]$ such that it is in $N_{A_{m+2}}^{\R_{m+2}}$. Then by \cite[Theorem 2.14]{JacobsonAlgebraBook85}, there exists polynomials $Z(t)$ and $R(t)$ with either $R(t)=0$ or $\deg(R(t)) < \deg(F_{m+2}(t))$ such that $$G(t) = Z(t)F_{m+2}(t) + R(t).$$ 
Now, if possible, let $R(t)$ be non-zero. Then $R(t) \in N_{A_{m+2}}^{\R_{m+2}}$ and $\deg(R(t))<\deg(F_{m+2}(t))$. We consider the following two cases. 
        
If possible, suppose at least one of the coefficients of $R(t)$ is a unit in $\R$. 
Then, after reduction, $\theta(R(t)) = {r(t)}$ will be an annihilating polynomial of $\overline{A}\in \M_n(k)$ in $k[t]$. This contradicts the minimality of $\deg\overline{F}(t)$.
So, this case is not possible.
        
Thus, we assume that $R(t) \in \mathfrak{m}_{m+2}[t]$.  
Since $\R_{m+2}$ is a local principal ideal ring, we may write $\mathfrak{m}_{m+2} = \langle \pi_{m+2} \rangle$.  
Then there exists $S(t)\in \mathscr{O}_{m+2}[t]$ with leading coefficient not congruent to 0 $\pmod{ \ker(\theta_{m+1})}$ such that $R(t) = \pi_{m+2}\cdot S(t)$. Therefore, $\deg R(t) = \deg S(t)$.  
Since $R(t) \in N_{A_{m+2}}^{\R_{m+2}}$, we have $\mathbf{0} = R(A_{m+2}) = \pi_{m+2} S(A_{m+2})$, and hence $S(A_{m+2}) \in \M_n(\mathfrak{m}_{m+1})$.  
Applying $\theta_{m+1}$ gives $$\theta_{m+1}({S}(A_{m+2})) = \mathbf{0}$$ which further implies that $\theta_{m+1}(S)(A_{m+1})=\mathbf{0}$. 
Therefore $\theta_{m+1}(S(t))\in N_{A_{m+1}}^{\R_{m+1}}$. Now $$\deg(\theta_{m+1}(S(t)))=\deg(S(t))=\deg(R(t))<\deg(F_{m+2}(t))=\deg(F_{m+1}(t)),$$
since $F(t)$ is monic.
As $N_{A_{m+1}}^{\R_{m+1}}=\langle F_{m+1}(t)\rangle$ by induction hypothesis, it contradicts the fact that $\theta_{m+1}(S(t))\in N_{A_{m+1}}^{\R_{m+1}}$.
Hence, the only possibility is that $R(t)$ must be the zero polynomial. 
Therefore $N_{A_{m+2}}^{\R_{m+2}}=\langle F_{m+2}(t)\rangle$.
\end{proof}

\begin{corollary}{\label{char poly and null ideal}}
    If $A\in \GL_n(\mathscr{\widehat O})$ is cyclic then \Cref{Null ideal} holds for that $A$. Moreover in this case, $\chi(A_{j+1})(t)=\mathrm{M in}_{\mathscr{O}_{j+1},A_{j+1}}(t)$ for each $j\in \mathbb{N}$. 
\end{corollary}
\begin{proof}
   Fix one $j$. By \Cref{Null ideal}, $N^{\R_{j+1}}_{A_{j+1}}= \langle \chi(A_{j+1})(t)\rangle$ (assuming $F(t)=\chi(A)(t))$. 
   Now consider $H(t)\in\mathscr{O}_{j+1}[t]$ be another monic polynomial that annihilates $A_{j+1}$ and $\deg(H)<\deg(\chi(A_{j+1}))$. 
   Then from the description of the null ideal, there exists $P(t)\in\mathscr{O}_{j+1}[t]$ such that $H(t)=P(t)\chi(A_{j+1})(t)$. But $\deg(P(t)\chi(A_{j+1})(t))$ is always $\geq\deg(\chi(A_{j+1})(t))$. 
   This contradicts the existence of $H(t)$ in $N^{\R_{j+1}}_{A_{j+1}}$ with $\deg(H)<\deg(\chi(A_{j+1}))$. Hence $\chi(A_{j+1})(t)=\mathrm{M in}_{\mathscr{O}_{j+1},A_{j+1}}(t)$. This proves the second assertion, whence the first assertion, by taking $F=\chi(A)(t)$.
\end{proof}

The {canonical map} $ \mathcal{T} \colon \widehat{\mathcal{O}} \rightarrow \varprojlim\limits_{j\geq 1}\frac{\widehat{\mathscr{O}}}{\mathfrak{m}^j\widehat{\mathscr{O}}} $ defined by

\[
\mathcal{T}(r) = \left( r + \mathfrak{m},\  r + \mathfrak{m}^2,\  r + \mathfrak{m}^3,\  r + \mathfrak{m}^4,\  \dots \right).
\]  is a ring homomorphism and $\widehat{\mathscr{O}}$ is \emph{complete with respect to maximal ideal} $\mathfrak{m}$ means $\mathscr{T}$ is an isomorphism; see \cite[section 7.1]{Eisenbud95} . 
In this case the isomorphism $\mathscr{T}$ induces an isomorphism  $\mathfrak{T}: \M_n(\widehat\O)\rightarrow \M_n(\varprojlim\limits_{j\geq 1}\frac{\widehat{\mathscr{O}}}{\mathfrak{m}^j\widehat{\mathscr{O}}})$ defined by $A=[a_{u,v}]\mapsto \mathfrak{T}(A)=[\mathscr{T}(a_{u,v})]$ where $u,v \in \{1,2,\cdots,n\}$ here. Using the description of the canonical map $\mathscr{T}$, we can further define a map $\mathfrak{T}^{*}:\M_n(\varprojlim\limits_{j\geq 1}\frac{\widehat{\mathscr{O}}}{\mathfrak{m}^j\widehat{\mathscr{O}}})\rightarrow \varprojlim\limits_{j\geq 1}\M_n(\frac{\widehat{\mathscr{O}}}{\mathfrak{m}^j\widehat{\mathscr{O}}})$ which takes $\mathfrak{T}(A)$ to $(A_1,A_2,\ldots)$; where $A_i=[a_{u,v}+\mathfrak{m}^i]$ for $i\in \mathbb{N}$;  (Note: $a_{u,v}$ denotes the $(u,v)$-th entry of matrix $A$).
This correspondence gives an isomorphism between $\M_n(\varprojlim\limits_{j\geq 1}\frac{\widehat{\mathscr{O}}}{\mathfrak{m}^j\widehat{\mathscr{O}}})$ and $\varprojlim\limits_{j\geq 1}\M_n(\frac{\widehat{\mathscr{O}}}{\mathfrak{m}^j\widehat{\mathscr{O}}})=\varprojlim\limits_{j\geq 1}\M_n(\mathscr{O}_j)$ and hence an isomorphism  $\mathscr{T}^{*}:\M_n(\widehat\O)\rightarrow \varprojlim\limits_{j\geq 1}\M_n(\mathscr{O}_j)$ defined as $A\mapsto \mathfrak{T}^{*}\circ\mathfrak{T}(A)$; which can be visualized by the following commutative diagram.

\[\begin{tikzcd}
	{\M_n(\widehat\O)} & {\M_n(\varprojlim\limits_{j\geq 1}\frac{\widehat{\mathscr{O}}}{\mathfrak{m}^j\widehat{\mathscr{O}}})} \\
	& {\varprojlim\limits_{j\geq 1}\M_n(\mathscr{O}_j)}
	\arrow["{\mathfrak{T}}", from=1-1, to=1-2]
	\arrow["{\mathscr{T}^{*}=\mathfrak{T}^{*}\circ\mathfrak{T}}"', from=1-1, to=2-2]
	\arrow["{\mathfrak{T}^{*}}", from=1-2, to=2-2]
\end{tikzcd}\]
\begin{theorem}{\label{lift-length two}}
    Let $\widehat{\mathscr{O}}$ be a local principal ideal ring, complete with respect to its unique maximal ideal $\mathfrak{m}$ and having residue field $k$ of characteristic not equal to 2. Let $A$ be a cyclic element in $\M_n(\widehat{\mathscr{O}})$ and $F\in\widehat{\mathscr{O}}[t]$ be a monic polynomial of degree $d$.
Let there be $\widetilde{B}\in\M_n(k)$ such that $f(\widetilde{B})=\overline{A}$ and $f'(\widetilde{B})\in\GL_n(k)$. 
Then there exists $B\in\M_n(\widehat{\mathscr{O}})$ such that $\overline{B}=\widetilde{B}$ and $F(B)=A$.
\end{theorem}
\begin{proof}
   By \Cref{prop:hens-root-full}, there exists $B_2\in M_n(\mathscr{O}_2)$ such that
$F_2(B_2)=A_2$.
We construct such a matrix by successively lifting a solution of the equation
\[
F_j(X)=A_j
\]
from the level $\mathscr{O}_j$ to $\mathscr{O}_{j+1}$ for all $j\ge 2$.
This is carried out by induction.
We begin by proving the following statement by induction on $i$:

\noindent \emph{\textbf{Claim 1:} For each $i\in \mathbb{N}$, any solution $B_i\in \M_n(\mathscr{O}_i)$ of the equation
$F_i(X)=A_i$
satisfying $F_i'(B_i)\in \GL_n(\mathscr{O}_i)$ can be lifted to a solution
$B_{i+1}\in \M_n(\mathscr{O}_{i+1})$
of the equation
$F_{i+1}(X)=A_{i+1}.$
}

 \emph{Proof of the claim 1:}  Start with $i=1$. By \Cref{prop:hens-root-full} we get the existence of $B_2\in \M_2(\mathscr{O}_2)$ such that $F_2(B_2)=A_2$. Hence, the statement is true for $i=1$.
   
   \textbf{Induction Hypothesis:} Assume that the statement is true for all $i\leq m-1$.

   \textbf{Inductive Step:} Consider the equation $F_{m+1}(X)=A_{m+1}$ for $A_{m+1}\in\M_n(\mathscr{O}_{m+1})$ after evaluating matrix indeterminate $X$ instead of $t$ in $F_{m+1}(t)$.
   Applying the reduction $\theta_{m}$ we obtain the equation $F_{m}(X)=A_m$ over $\M_n(\mathscr{O}_m)$ after evaluating matrix indeterminate $X$ instead of $t$ in $F_{m}(t)$. 
   By induction hypothesis, there exists $B_m\in \M_n(\mathscr{O}_m)$ such that $F_m(B_m)=A_m$ and $F_m'(B_m)\in \GL_n(\mathscr{O}_m)$, as $F_{m+1}'(B_{m+1})\in \GL_n(\mathscr{O}_{m+1})$.
   Note that $B_m\in \mathscr{Z}_{\M_n(\mathscr{O}_m)}(A_m)$. 
   If $F_{m+1}(X)=A_{m+1}$ has a solution in $\M_n(\mathscr{O}_{m+1})$ then it must belong to $\mathscr{Z}_{\M_n(\mathscr{O}_{m+1})}(A_{m+1})$. 
   There is a ring isomorphism between $\mathscr{Z}_{\M_n(\mathscr{O}_{m+1})}(A_{m+1})$ and  $\dfrac{\mathscr{O}_{m+1}[t]}{\langle\chi(A_{m+1})(t)\rangle}$; see  \cite[section 4]{PRS2025}. Define the evaluation map $\mathrm{EV}_{m+1}: \mathscr{O}_{m+1}[t]\rightarrow \mathscr{O}_{m+1}[A_{m+1}]$ by $H(t)\mapsto H(A_{m+1})$. This is a surjection with $\ker(\mathrm{EV}_{m+1})=N^{\R_{m+1}}_{A_{m+1}}$. 
   Since $N^{\R_{m+1}}_{A_{m+1}}=\langle\chi(A_{m+1})(t)\rangle$ by \Cref{Null ideal} and \Cref{char poly and null ideal}, one must have $\dfrac{\mathscr{O}_{m+1}[t]}{\langle\chi(A_{m+1})(t)\rangle}\cong \mathscr{O}_{m+1}[A_{m+1}]$. 
   Similarly $\dfrac{\mathscr{O}_{m}[t]}{\langle\chi(A_{m})(t)\rangle}\cong \mathscr{O}_{m}[A_{m}]$.
   
   \[\begin{tikzcd}
	{\mathscr{Z}_{\M_n(\mathscr{O}_{m+1})}(A_{m+1})} && {\mathscr{Z}_{\M_n(\mathscr{O}_{m})}(A_{m})} \\
	{\mathscr{O}_{m+1}[A_{m+1}]} && {\mathscr{O}_{m}[A_{m}]}
	\arrow["{\theta_m}", from=1-1, to=1-3]
	\arrow["\mathrm{iso}_{m+1}"', from=1-1, to=2-1]
	\arrow["\mathrm{iso}_m", from=1-3, to=2-3]
	\arrow["{\theta_m}", from=2-1, to=2-3]
\end{tikzcd}\]

Let $B_m\in\Zen_{\M_n(\R_m)}(A_m)$, and $\mathrm{iso}_m(B_m)=\sum\limits_{j=0}^ra_jA_m^j$. 
Since $\theta_m$ is a surjection, there exists $\check{B}_{m+1}=\sum\limits_{j=0}^rb_jA_{m+1}^j\in\mathscr{O}_{m+1}[A_{m+1}] $ such that $\theta_m(\check{B}_{m+1})=\is_m(B_m)$.
Note that $\ker(\theta_m)=\mathfrak{n}_{m+1}[A_{m+1}]$. where $\mathfrak{n}_{m+1}=\ker(\mathscr{O}_{m+1}\rightarrow\mathscr{O}_m)=\dfrac{\pi^m\widehat{\O}}{\pi^{m+1}\widehat{\R}}$.

As $\theta_m(F_{m+1}(\check{B}_{m+1}))=F_m(\is_m(B_m))=\is_m(A_m)=A_m$,
one can write $F_{m+1}(\check{B}_{m+1})= A_{m+1}+\pi_{m+1}^mC$ for some $C\in \mathscr{O}_{m+1}[A_{m+1}]$. 
As $F_m'(B_m)\in \GL_n(\mathscr{O}_m)$ by induction hypothesis, therefore $F_m'(\is_m(B_m))\in \mathscr{O}_m[A_m]^{\times}$.
As $\theta_m(F_{m+1}'(\check{B}_{m+1}))=F_m'(\is_m(B_m))$ therefore $\is_{m+1}(F_{m+1}'(\check{B}_{m+1}))$ must be an unit in $\mathscr{O}_{m+1}[A_{m+1}]$.
Consider $D=-CF_{m+1}'(\check{B}_{m+1})^{-1}$. Invoking the Taylor series expansion (See for instance \cite{ConradHensel}), we obtain \begin{align*}
    F_{m+1}(\check{B}_{m+1}+\pi_{m+1}^mD)&=F_{m+1}(\check{B}_{m+1})+\pi_{m+1}^mDF_{m+1}'(\check{B}_{m+1})\\
    &=F_{m+1}(\check{B}_{m+1})-\pi_{m+1}^mC\\
    &=A_{m+1}.
\end{align*} 
Note that $\widehat{B}_{m+1}=\check{B}_{m+1}+\pi_{m+1}^mD$ is a lift of $\is_m(B_m)$ with respect to $\theta_m:\mathscr{O}_{m+1}[A_{m+1}]\rightarrow\mathscr{O}_m[A_m]$.
Through the isomorphism $\is_{m+1}$, there exists a unique $B_{m+1}\in \mathscr{Z}_{\M_n(\mathscr{O}_{m+1})}(A_{m+1})$ such that $\is_{m+1}(B_{m+1})=\widehat{B}_{m+1}$. 
The commutativity of the above diagram and existence of the isomorphisms $\is_{m+1}, \is_m$ with $\is_{m+1}(A_{m+1})=A_{m+1}$ and $\is_m(A_m)=A_m$, immediately implies that, $B_{m+1}$ must be a lift of $B_m$ with respect to $\theta_m: \mathscr{Z}_{\M_n(\mathscr{O}_{m+1})}(A_{m+1})\rightarrow\mathscr{Z}_{\M_n(\mathscr{O}_{m})}(A_{m})$ such that $F_{m+1}(B_{m+1})=A_{m+1}$. Hence, claim 1 is true for all $i\in \mathbb{N}$.

Now the canonical isomorphism $\mathscr{T}$ induces an isomorphism $\widetilde{\mathscr{T}}$ in the polynomial ring level given by:
$$\widetilde{\mathscr{T}}:\widehat{\mathscr{O}}[t]\rightarrow \varprojlim\limits_{j\geq 1}\frac{\widehat{\mathscr{O}}}{\mathfrak{m}^j\widehat{\mathscr{O}}}[t]$$ defined by $H(t)\mapsto (H_1(t),H_2(t),\cdots)$ where $H_j(t)=\widehat{\theta}_j(H(t))\in \frac{\widehat{\mathscr{O}}}{\mathfrak{m}^j\widehat{\mathscr{O}}}[t] $. 
This, and by the commutative diagram before the start of the theorem, we obtain the following commutative diagram.

\[\begin{tikzcd}
	{\M_n(\widehat\O)} && {\varprojlim\limits_{j\geq 1}\M_n(\mathscr{O}_j)} \\
	{\M_n(\widehat\O)} && {\varprojlim\limits_{j\geq 1}\M_n(\mathscr{O}_j)}
	\arrow["{\mathscr{T}^*}", from=1-1, to=1-3]
	\arrow["H"', from=1-1, to=2-1]
	\arrow["{(H_1,H_2,\cdots)}", from=1-3, to=2-3]
	\arrow["{\mathscr{T}^*}"', from=2-1, to=2-3]
\end{tikzcd}\]
where $H(X)$ for $X\in \M_n(\widehat{\mathscr{O}})$ denotes evaluation at $X$ in place of $t$ in $H(t)$.
A similar notion holds for $H_i$'s.

Let $A\in \M_n(\widehat{\O})$ be cyclic such that there exists  $\Tilde{B}\in \M_n(k)$ be such that $f(\Tilde{B})=\Bar{A}$ and $f'(\Tilde{B})\in \GL_n(k)$ where $A_1=\Bar{A}$ and $B_1=\Tilde{B}$. 
By \Cref{prop:hens-root-full}, there exists $B_2\in \M_n(\O_2)$ such that $F_2(B_2)=A_2$ and $F_2'(B_2)\in \GL_n(\O_2)$.
By \emph{Claim 1} we get the existence of $B_3,B_4,\cdots$ satisfying the properties mentioned in \emph{Claim 1} such that $F_j(B_j)=A_j$ for each $j\in \mathbb{N}$. 
Now under the isomorphism $\mathscr{T}^*$, we have $\mathscr{T}^*(A)=(A_1,A_2,\cdots)$. As $(B_1,B_2,\cdots)\in \varprojlim\limits_{j\geq 1}\M_n(\mathscr{O}_j)$ by \emph{Claim 1} (as $B_{j+1}$ is the lift of $B_j$ for $j\geq 1)$. 
Therefore there exists $B\in \M_n(\widehat{\O})$ such that $\mathscr{T}^*(B)=(B_1,B_2,\cdots)$.
Therefore $B$ is a lift of $B_1=\Tilde{B}$ under the map $\widehat{\theta}_1: \M_n(\widehat{\O})\rightarrow \M_n(k)$.
By the above commutative diagram, we obtain $F(B)=A$. 
\end{proof}
\color{black}


\begin{theorem}\label{thm:n-var}
     Let $\widehat{\mathscr{O}}$ be a local principal ideal ring, complete with respect to its unique maximal ideal $\mathfrak{m}$ having residue field $k$ of characteristic other than $2$. Let $A$ be a cyclic element in $\M_n(\widehat{\mathscr{O}})$ and $F\in\widehat{\mathscr{O}}[x_1,x_2,\cdots,x_m]$ be a monic polynomial of degree $d$.
Let there be $\mathbf{\widetilde{B}}=(\widetilde{B}_1,\widetilde{B}_2,\cdots,\widetilde{B}_m)\in\M_n(k)\times\M_n(k)\times\cdots\times\M_n(k) $ ($m$-copies) such that $\widetilde{B}_i,\widetilde{B}_j$ commutes with each other for every $i,j\in \{1,2,\cdots,m\}$ and $f(\mathbf{\widetilde{B}})=\overline{A}$. Moreover we assume that $\frac{\partial f}{\partial x_i}|_{\mathbf{\widetilde{B}}}\in\GL_n(k)$ for $i=1,2,\cdots,r$; where $1\leq r\leq m$ and $\frac{\partial f}{\partial x_i}|_{\mathbf{\widetilde{B}}}=\mathbf{0}\in \M_n(k)$ for $i=r+1,r+2,\cdots,m$. 
Then there exists $\mathbf{B}=(B_1,B_2,\cdots,B_m)\in\M_n(\widehat{\mathscr{O}})\times\M_n(\widehat{\mathscr{O}})\times...\times\M_n(\widehat{\mathscr{O}})$ ($m$-copies) such that $\overline{\mathbf{B}}=\mathbf{\widetilde{B}}$ (i.e. $\widehat{\theta}_1(B_i)=\widetilde{B}_i$ for each $i=1,2,\cdots,m$) and $F(\mathbf{B})=A$.
\end{theorem}
\begin{proof}
At first, we prove this for $\ell=2$. 
Later, we apply an induction argument to ensure its validity over each finite step, and from there we prove the theorem for $\widehat{\mathscr{O}}$ level.
\[\begin{tikzcd}
	{(\mathscr{Z}_{\M_n(\mathscr{O}_{2})}(A_{2}))^m} && {(\mathscr{Z}_{\M_n(k)}(\Bar{A}))^m} \\
	{(\mathscr{O}_2[A_2])^m} && {(k[\Bar{A}])^m}
	\arrow["{(\theta_1,\theta_1,\cdots,\theta_1)}", from=1-1, to=1-3]
	\arrow["\mathbf{\mathrm{ISO}_{2}}"', from=1-1, to=2-1]
	\arrow["\mathbf{\mathrm{ISO}_1}", from=1-3, to=2-3]
	\arrow["{(\theta_1,\theta_1,\cdots,\theta_1)}", from=2-1, to=2-3]
\end{tikzcd}\]  where $\mathbf{\mathrm{ISO}_{2}}=\underbrace{(\mathrm{iso}_2,\mathrm{iso}_2,\cdots,\mathrm{iso}_2)}_{m-\text{times}}$ and $\mathbf{\mathrm{ISO}_{1}}=\underbrace{(\mathrm{iso}_1,\mathrm{iso}_1,\cdots,\mathrm{iso}_1)}_{m-\text{times}}$.

We readily have $\mathrm{ISO}_2(A_2,A_2,\cdots,A_2)=(A_2,A_2,\cdots,A_2)$.
We consider $\widehat{\theta}_2(F)=F_2$ over $\R_2$. 
Given $\widetilde{\mathbf{B}} \in\GL_n(k)$ satisfying $\overline{F}(\widetilde{\mathbf{B}})=\overline{A}$, we have $\widetilde{B}_i\in\Zen_{\M_n(k)}(\overline{A})$ for each $i=1,2,\cdots,m$.
Since $\overline{A}$ is cyclic, $\Zen_{\M_n(k)}(\overline{A}) = k[\overline{A}]$, thus
$\widetilde{B}_i = \sum\limits_{j=0}^sa_j^{(i)} \overline{A}^j$ for some $a_j^{(i)}\in k$. 
Now denote $\sum\limits_{j=0}^s b_j^{(i)} A_2^j = \check{B}_i \in\R_2[A_2] \subseteq\M_n(\R_2)$, a lift of $\widetilde{B}_i$, under the map $\theta_1 \colon \R_2[A_2] \longrightarrow k[\overline{A}]$, where $\theta(b_j^{(i)})=a_j^{(i)}$. 
Note that $\ker(\theta_1) = \mathfrak{m}[A_2]$.

Consider $\mathbf{\check{B}}=(\check{B}_1,\check{B}_2,\cdots,\check{B}_m)$. 
We have $F_2(\mathbf{\check{B}})=A_2+\pi C$ for some $C\in \R_2[A_2]$.
We want to produce $\mathbf{B_0}=(\check{B}_1+\pi D_1,\check{B}_2+\pi D_2,\cdots,\check{B}_m+\pi D_m)$ so that $F_2(\mathbf{B_0})=A_2$ for some $D_i\in \R_2[A_2]$; $i=1,2,\cdots,m$. 
Over any commutative ring $\mathscr{R}$ with unity, and for a polynomial $F(\textbf{x})$ in $m$-variables over $\mathscr{R}$, one can write $F(\textbf{x}+\textbf{y}) = F(\textbf{x}) + \sum_{i=1}^{m} \frac{\partial F}{\partial x_i} y_i + \sum_{1 \leq i,j \leq m} C_{ij}(\textbf{x},\textbf{y}) y_i y_j$; where $C_{ij}(\textbf{x},\textbf{y})\in \mathscr{R}[\textbf{x},\textbf{y}]$; $\mathbf{x}=(x_1,x_2,\cdots,x_m)$ and $\mathbf{y}=(y_1,y_2,\cdots,y_m)$; see \cite[equation (3.4)]{ConradmultiHensel}.
Evaluating $\check{B}_i$ and $\pi D_i$ in place of $x_i$ and $y_i$ respectively for every $i=1,2,\cdots,m$ we obtain $F_2(\mathbf{B_0})=F_2(\check{\mathbf{B}})+\pi \sum_{i=1}^{m} \frac{\partial F_2}{\partial x_i}|_{\check{\mathbf{B}}} D_i $.
One needs to pick $D_i$s carefully so that $F_2(\mathbf{B_0})=A_2$. 
Here, two cases may arise depending on $r$.

\noindent \textbf{Case I. $(p \nmid r)$}: In this case let us choose $D_i= -C\left(r\frac{\partial F_2}{\partial x_i}|_{\check{\mathbf{B}}}\right)^{-1}$ for $i=1,2,\cdots,r$ and any arbitrary element of $\R_2[A_2]$ as $D_i$'s for $i=r+1,r+2,\cdots,m$.

\noindent \textbf{Case II. $(p \mid r)$}:  Choose $$D_i=\begin{cases}
    -C\left(2(r-1)\frac{\partial F_2}{\partial x_i}|_{\check{\mathbf{B}}}\right)^{-1}&1\leq i\leq r-1\\
    -C\left(2\frac{\partial F_2}{\partial x_i}|_{\check{\mathbf{B}}}\right)^{-1}&i=r
\end{cases}$$
and for $i=r+1,r+2,r+3,\cdots,m$, any element of $\R_2[A_2]$ fits well as a choice of $D_i$ because then $\frac{\partial F_2}{\partial x_i}|_{\check{\mathbf{B}}}$ are non-units in $\R_2[A_2]$. Therefore, the statement for $\ell=2$ holds. 
Now, by a similar sort of induction method used in \Cref{lift-length two}, the result follows for the case of $\widehat{\R}$, taking into account that the choice of $\mathbf{B}_0$ should be $(\check{B}_1+\pi^\ell D_1,\cdots,\check{B}_m+\pi^\ell D_m)$ when one considers lifting $\M_n(\R_\ell)\longleftarrow\M_n(\R_{\ell+1})$.
\end{proof}
\begin{example}
    Consider $\widehat{\R}=\mathbf{Z}_5$ the ring of $5$-adic integers. 
    Its residue field $k$ is isomorphic to $\mathbb{F}_5$.
    Take $A=\left(\begin{array}{cc}
       0  &  0\\
        0 & 9
    \end{array}\right)\in \M_2(\mathbf{Z}_5)$ to be a lift of $\Bar{A}=\left(\begin{array}{cc}
       0  &  0\\
        0 & 4
    \end{array}\right)\in \M_2(\mathbb{F}_5)$. 
    Consider $F(x,y)=xy+y^2\in \mathbf{Z}_5[x,y]$.
    Now there exists $\Tilde{B}=(\Tilde{B}_1,\Tilde{B}_2)\in (\M_2(\mathbb{F}_5))^2$ such that it satisfies assumptions of \Cref{thm:n-var}. 
    Take $\Tilde{B}_1=\left(\begin{array}{cc}
       4  &  0\\
        0 & 0
    \end{array}\right)$, $\Tilde{B}_2=\left(\begin{array}{cc}
       1  &  0\\
        0 & 2
    \end{array}\right)$. Consider $\Bar{F}=f$, then $\frac{\partial f}{\partial x}|_{\Tilde{B}}=\Tilde{B}_2, \frac{\partial f}{\partial y}|_{\Tilde{B}}=\Tilde{B}_2^2 $ both are invertible. There exists $B=(B_1,B_2)\in (\M_2(\mathbf{Z}_5))^2$ such that $F(B_1,B_2)=A$ where $B_1=\left(\begin{array}{cc}
       -1  &  0\\
        0 & 2^{-1}.5
    \end{array}\right), B_2=\left(\begin{array}{cc}
       1  &  0\\
        0 & 2
    \end{array}\right)$ are the lifts of $\Tilde{B}_1$ and $\Tilde{B}_2$ respectively, in $\M_2(\mathbf{Z}_5)$.
\end{example}

\begin{remark}[A few words on word problems on groups]
    The word problems on group is well-known; given an element $w\in \mathbf{F}_r$ and a group $G$, one gets a map by evaluation
    \begin{align*}
        \widetilde{w}:G^r\longrightarrow G,\,(g_1,g_2,\cdots,g_r)\mapsto w(g_1,g_2,\cdots,g_r);
    \end{align*}
    known as \emph{word maps}. 
    Influential works include the settlement of the Ore's conjecture and that of the width of a word on finite simple groups; see \cite{LiebeckBrienShalev10,Shalev2009}.
    In the context of lifting solution, one has the result for the commutator map on $\mathrm{SL}_n(\R)$ in \cite{AvniShalev13}.
    \Cref{thm:n-var} also shows that if one starts with a word $w$, having no inverse in its reduced form, then lifting of a solution to the equation $\widetilde{w}(x_1,x_2,\ldots,x_r)=A$ for a cyclic matrix $A\in\GL_n(\widehat{\R})$ exists whenever a pairwise commuting solution exists for the equation $\widetilde{w}(x_1,x_2,\cdots,x_r)=\overline{A}$, an equation in $\GL_n(k)$; indeed $\GL_n(\R)\subseteq \M_n(\R)$ and $A\in\GL_n(\R)$ iff $\overline{A}\in \GL_n(k)$.

\end{remark}

\printbibliography
\vspace{2em}
\end{document}